\documentclass[a4paper,11pt]{article}
\usepackage[all]{xy}
\usepackage[T1]{fontenc}
\usepackage[utf8]{inputenc}
\usepackage{lmodern}
\usepackage{microtype}
\usepackage{mathtools}
\usepackage{amsfonts}
\usepackage{amsmath}
\usepackage{amssymb}
\usepackage{pictexwd,dcpic}
\usepackage{thmtools}
\newtheorem{theorem}{Theorem}[section]
\newtheorem{lemma}[theorem]{Lemma}
\newtheorem{proposition}[theorem]{Proposition}
\newtheorem{corollary}[theorem]{Corollary}

\newenvironment{proof}[1][Proof]{\begin{trivlist}
\item[\hskip \labelsep {\bfseries #1}]}{\end{trivlist}}
\newenvironment{definition}[1][Definition]{\begin{trivlist}
\item[\hskip \labelsep {\bfseries #1}]}{\end{trivlist}}

\newenvironment{remarks}[1][Remarks]{\begin{trivlist}
\item[\hskip \labelsep {\bfseries #1}]}{\end{trivlist}}
\newenvironment{remark}[1][Remark]{\begin{trivlist}
\item[\hskip \labelsep {\bfseries #1}]}{\end{trivlist}}

\newenvironment{conjecture}[1][Conjecture]{\begin{trivlist}
\item[\hskip \labelsep {\bfseries #1}]}{\end{trivlist}}

\newcommand{\qed}{\nobreak \ifvmode \relax \else
     \ifdim\lastskip<1.5em \hskip-\lastskip
      \hskip1.5em plus0em minus0.5em \fi \nobreak
      \vrule height0.75em width0.5em depth0.25em\fi}
\usepackage[top=3cm,bottom=3cm,left=2.5cm,right=3cm]{geometry}

 \DeclareFontFamily{U}{wncy}{}
    \DeclareFontShape{U}{wncy}{m}{n}{<->wncyr10}{}
    \DeclareSymbolFont{mcy}{U}{wncy}{m}{n}
    \DeclareMathSymbol{\Sh}{\mathord}{mcy}{"58}
\usepackage{hyperref}
\usepackage{fancyhdr}
\date{}
  \title{\textbf{On a Gross conjecture over imaginary quadratic fields}}
  \author{Saad El Boukhari\footnote{S. El Boukhari, Moulay Ismail University of Meknès, Department of Math., B.P. 11201 Zitoune, Meknès, Morocco.
\textit{E-mail:} saadelboukhari1234@gmail.com}}
\begin{document}
\maketitle
\date{}
\begin{abstract}
Let $k$ be an imaginary quadratic number field, and $F/k$  a finite abelian extension of  Galois group $G$.
We show that a Gross conjecture concerning the leading terms of Artin $L$-series holds for $F/k$ and all rational primes which are split in $k$ and which do not divide $6$.
\end{abstract}


\section{Introduction}
Let $F$ be a number field and $O_F$ its ring of integers. Let $r$ denote a strictly negative integer. The Quillen higher algebraic $K$-groups of the ring of integers $O_F$  have the following property :\\
\textbf{(Borel)}
\textit{The groups $K_{-2r}(O_F)$ are finite and the groups $K_{1-2r}(O_F)$ are of finite type over $\mathbb{Z}$}

$$ \mathrm{rank}_{\mathbb{Z}}(K_{1-2r}(F)):= d_{r} =  \left\{
\begin{array}{rl}
r_{1}+r_{2} \mathrm{\;\;\;\; if\; r \; is\; even} \\
r_{2} \;\mathrm{\;\;\;\; if\; r \; is\; odd} 
\end{array}
\right.$$
\noindent \textit{where $r_{1}$ and $r_{2}$ denote respectively the number of real and  complex places of $F$.}\\
It is known in literature that for $r<0$, the odd degree $K$-groups $K_{1-2r}(O_F)$
are analogues of the unit group $O^{*}_F$ of $O_F$ and the even degree finite $K$-groups $K_{-2r}(O_F)$ are analogues of the class group $\mathrm{Cl}_F$ of $F$.\\
Suppose that the number field $F$ is a finite abelian extension of some number field $k$ and let $G:=\mathrm{Gal}(F/k)$. We also write $\widehat{G}:=\mathrm{Hom}(G,\mathbb{C}^{\times})$ for the (multiplicative) group of one dimensional (irreducible) complex characters of $G$. Let $\mathrm{Ram}(F/k)$ denote the set of finite places of $k$ which ramify in $F/k$ and $S_{\infty}$ the set of archimedean places of $k$. If
$\chi\in\widehat{G}$, the Artin $L$-function attached to
$\chi$  is defined for $\{s\in \mathbb{C}:\; \mathrm{Re}(s)>1\}$ by
\begin{equation*}
L(s,\chi)=\prod_{\mathfrak{p}\not\in \mathrm{Ram}(F/k)\cup S_\infty}(1-\chi(\mathrm{Frob}_{\mathfrak{p}})\mathrm{N}\mathfrak{p}^{-s})^{-1}
\end{equation*}
where $\mathrm{Frob}_{\mathfrak{p}}\in G$ is the Frobenius of the
(unramified) prime $\mathfrak{p}$. This function can be analytically
continued to a meromorphic function on $\mathbb{C}$.\\

Gross formulated in the 1970s a  conjecture which asserts that the leading term at any strictly negative integer of the Artin $L$-function  should be equal to within an undetermined algebraic factor to the image by the Beilinson regulator (see Subsection \ref{BeilinsonRegSec} below) of a certain element in odd degree higher algebraic $K$-groups (Gross's conjecture was published more recently in \cite{GrossConj}). Gross's conjecture is the natural analogue in higher algebraic $K$-theory and strictly negative integers of Stark's Conjecture (e.g. \cite{Tate1}) which provides for a certain appropriate set $S$ of primes of $k$ special $S$-units linked to the special value at $s=0$ of the $L$-function $L(s,\chi)$. \\
 In the absolute abelian case (i.e. $k=\mathbb{Q}$), Gross's conjectural units are explicitly given by Beilinson's elements \cite{beilinson1} and $p$-adically by the Beilinson-Deligne-Soulé elements (e.g.   \cite{Soule}). Their applications in this case are many and touch various subjects in the number theory of higher algebraic $K$-groups (e.g. \cite{KNF}, \cite{SE1}, \cite{SE2}  and others). In the non-absolute abelian case, Gross's conjecture doesn't provide any particular formula for the definition of these special units.\\
The next natural step is to look for these special units in the "almost" abelian case, i.e. the case where $k$ is an imaginary quadratic field. This is exactly the aim of this work: In Section \ref{section1} we provide a definition of Gross's elements over an imaginary quadratic field using the Beilinson regulator. In Section \ref{section5} we show that this latter is a consequence of the Equivariant Tamagawa Number conjecture (or the ETNc for short, which is formulated in Section \ref{section2}). Then we consider the case in which the ETNc is a theorem to prove Gross's conjecture for $F/k$.

\subsection{Notations}
The notations are:

\begin{itemize}
\item $F/k$ is a finite abelian extension of number fields with $k$ imaginary quadratic, and $G:=\mathrm{Gal}(F/k)$. 
\item $\hat{G}:=\mathrm{Hom}(G,\mathbb{C}^{\times})$ is the (multiplicative) group of one dimensional (irreducible) complex characters of the finite group $G$.
\item $E$ is the number field generated over $\mathbb{Q}$ by all values of characters $\chi\in\hat{G}$, and $O$ its ring of integers.
\item  $r<0$ is a rational integer.
\item  $E(r)_{F}$ is the (Tate) motive: $E(r)_{F}:=h^{0}(\mathrm{Spec}(F))_{E}(r).$
\item $p$ a rational prime number.
\item $E_p:=E\otimes_{\mathbb{Z}}\mathbb{Z}_p$ and $O_p:=O\otimes_{\mathbb{Z}}\mathbb{Z}_p$.
\item $S$ is a finite set of places of $k$ containing the set $S_p$ (of $p$-places and infinite places) and all the finite places which ramify in $F/k$. Let $S_f$ be the subset of finite places of $S$.
\item If $A$ is any finite set, we write ${\mid}A{\mid}$ for its cardinal (i.e. the number of its elements). 
\item For any finite set $S^{'}$ of finite places of $k$
we write $\varepsilon_{S^{'}}(r):=\prod_{v\in S^{'}}(1-\mathbf{N}v^{-r}\mathrm{Frob}_{v}e_{I_{v}})$.\\
where $\mathrm{Frob}_{v}$ is any representative in $G$ of the Frobenius map  associated to the finite place $v$, and $e_{I_{v}}$ is the idempotent of $\mathbb{Q}[G]$ which corresponds to the inertia subgroup $I_v$ $$e_{I_{v}}=\frac{\sum_{\sigma\in I_v}\sigma}{{\mid}I_v{\mid}}$$

\item For any number field $K$, we write $\overline{K}$ for a fixed algebraic closure of $K$ and $O_K$ for the ring of integers of $K$. For any finite set $S^{'}$ of places of $K$ containing the set $S_\infty$ of archimedean places of $K$, we denote by $O_{K,S^{'}}$ the ring of $S^{'}$-integers of $K$ (i.e. the elements of $K$ whose $\mathfrak{p}$-adic valuation is positive for all $\mathfrak{p}\not\in S^{'}$).
\item If $A$ is a finite $\mathrm{Gal}(\overline{\mathbb{Q}}/\mathbb{Q})$-module and $i$ is a rational integer, we write $A(i)$ for the "Tate twist" of $A$. This is the group $A$ endowed with the modified $\mathrm{Gal}(\overline{\mathbb{Q}}/\mathbb{Q})$-action $g.x:=\chi_{cyc}(g)^{i}g(x)$ for all $g\in \mathrm{Gal}(\overline{\mathbb{Q}}/\mathbb{Q})$ and $x\in A$ where $\chi_{cyc}$ is the cyclotomic character (e.g. \cite{NSW}, Definition (7.3.6)).
\item If $A$ is an $O_p$ module (resp. $E$-module, $E_p$-module) we let $A^{\vee}$ denote $\mathrm{Hom}(A, O_p)$ (resp. $\mathrm{Hom}(A, E)$, $\mathrm{Hom}(A, E_p)$). If $X$ is a complex of $O_p$ modules (resp. $E$-modules, $E_p$-modules), we define similarly $X^\vee$ by replacing $\mathrm{Hom}(X,\;)$ above by $\mathrm{RHom}(X,\;)$.
\end{itemize}
\section{Gross's conjecture in higher algebraic $K$-theory over an imaginary quadratic field}\label{section1}
\subsection{The Beilinson regulator}\label{BeilinsonRegSec}
Let $r<0$ be a rational integer. The Beilinson regulator defined over the $K$-group of the field of complex numbers (e.g. \cite{Rapoport} for more details) is a map 
$$\rho_{\mathbb{C}}^{r}: K_{1-2r}(\mathbb{C})\rightarrow H^{1}_{D}(\mathrm{Spec}(\mathbb{C}),(2\pi{i})^{1-r}\mathbb{R})\cong (2\pi{i})^{-r}\mathbb{R}$$
where $H^{1}_{D}()$ is the first group of Deligne's Cohomology and $i:=\sqrt{-1}$.

\noindent For a number field $F$ we compose with the map $$K_{1-2r}(O_F)=K_{1-2r}(F) \rightarrow \prod_{\sigma: F\rightarrow \mathbb{C}}K_{1-2r}(\mathbb{C}),$$
where the product is taken over all embeddings $\sigma \in\mathrm{Hom}(F,\mathbb{C})$ of $F$ in the field of complex numbers. We obtain a map
$$\rho_{F}^{r}: K_{1-2r}(O_F)\rightarrow  (X_{F}\otimes (2\pi{i})^{-r}\mathbb{R})^{+},$$
where $X_{F}:=\mathbb{Z}[\mathrm{Hom}(F,\mathbb{C})]$ (the free abelian group over the set $\mathrm{Hom}(F,\mathbb{C})$) and $(X_{F}\otimes (2\pi{i})^{-r}\mathbb{R})^{+}$ is the "plus part" of $(X_{F}\otimes (2\pi{i})^{-r}\mathbb{R})$ i.e. the submodule that is invariant under the action of complex conjugation.\\
The image of $\rho_{F}^{r}$ is a complete lattice of the $\mathbb{R}$-vector-space $(X_{F}\otimes (2\pi{i})^{-r}\mathbb{R})^{+}$. 
The kernel of the Beilinson regulator $\rho_{F}^{r}$ verifies
$$\mathrm{ker}(\rho_{F}^{r})\subseteq K_{1-2r}(O_F)_{\mathrm{tors}}$$
where $K_{1-2r}(O_F)_{\mathrm{tors}}$ is the $\mathbb{Z}$-torsion subgroup of $K_{1-2r}(O_F)$. Hence, if  we write $K_{1-2r}(O_F)_{/\mathrm{tors}}$ for the torsion-free quotient of $K_{1-2r}(O_F)$, then $\rho_{F}^{r}$ induces the following injective map:
\begin{equation*}
\xymatrix@=2pc{\rho_{F}^{r}:
K_{1-2r}(O_F)_{/\mathrm{tors}}\ar@{^{(}->}[r]&
(X_{F}\otimes (2\pi{i})^{-r}\mathbb{R})^{+}}
\end{equation*}
which becomes an isomorphism if we tensor the left hand side with $\mathbb{R}$.
\subsection{Special "Gross units" in higher algebraic $K$-groups over an imaginary quadratic field}
We regard the set $\mathrm{Hom}(F, \mathbb{C})$ of embeddings of $F$ in the field of complex numbers $\mathbb{C}$ as a left $G\times \mathrm{Gal}(\mathbb{C}/\mathbb{R})$-module by setting  $(g\times w)(\sigma)= w\circ \sigma\circ g^{-1}$, for all $g\in G$, $w\in \mathrm{Gal}(\mathbb{C}/\mathbb{R})$ and $\sigma\in \mathrm{Hom}(F, \mathbb{C})$.\\
Let $\mathbb{Z}[\chi]$ denote the ring of values of the character $\chi$.
By extension of scalars, we can rewrite the Beilinson regulator $\rho^{r}_{F}$ as
$$\rho^{r}_{F}: K_{1-2r}(O_F)_{/\mathrm{tors}}\otimes_{\mathbb{Z}}\mathbb{Z}[\chi]\hookrightarrow (\prod_{\sigma :F\rightarrow\mathbb{C}}(2\pi i)^{-r}\mathbb{R})^{+}\otimes_{\mathbb{Z}}\mathbb{Z}[\chi]$$
This map is still injective since $\mathbb{Z}[\chi]$ is a flat $\mathbb{Z}$-module.
For any character $\chi\in\hat{G}$, let $L_{S}(s,\chi)$ denote the $S$-truncated Artin $L$-function defined for $\{s\in \mathbb{C}:\; \mathrm{Re}(s)>1\}$ by
\begin{equation*}
L_{S}(s,\chi)=\prod_{\mathfrak{p}\not\in
S}(1-\chi(\mathrm{Frob}_{\mathfrak{p}})\mathbf{N}\mathfrak{p}^{-s})^{-1}
\end{equation*}
This function can be analytically
continued to a meromorphic function on $\mathbb{C}$.
We denote by $L_{S}^{'}(s,\chi)$ the first derivative of $L_{S}(s,\chi)$. We will use the following lemma:
\begin{lemma}\label{lemma3.1}
Let $k$ be imaginary quadratic. Then there is a canonical isomorphism of $\mathbb{R}[G]$-modules
$$ \iota :(\prod_{\sigma: F\rightarrow \mathbb{C}}(2\pi i)^{-r}\mathbb{R})^{+}\cong \mathbb{R}[G]$$
\end{lemma}
\begin{proof}
If $k$ is imaginary quadratic, then for each embedding $\hat{\sigma}\in \mathrm{Hom}(F, \mathbb{C})$, either $\hat{\sigma}$ or $\tau\hat{\sigma}$ identifies with an automorphism $\sigma\in G$.\\
The isomorphism $\iota$ is explicitly given by mapping,  for each such embedding $\hat{\sigma}\in \mathrm{Hom}(F, \mathbb{C})$, the element $(0,..,(2\pi i)^{-r}a_{\hat{\sigma}},..,(2\pi i)^{-r}a_{\tau\hat{\sigma}},..,0)$ (with $a_{\sigma}=\pm a_{\tau\sigma}$ depending on the parity of $r$) to $a_{\hat{\sigma}}\sigma^{-1}$ if $\hat{\sigma}=\sigma$ and to $a_{\tau\hat{\sigma}}\sigma^{-1}$ if $\tau\hat{\sigma}=\sigma$.\hspace*{\fill}$\Box$
\end{proof}
Using the Lemma above we can write the Beilinson regulator map as:
$$\rho^{r}_{F}: K_{1-2r}(O_F)_{/\mathrm{tors}}\otimes_{\mathbb{Z}}\mathbb{Z}[\chi] \hookrightarrow\mathbb{R}[G]\otimes_{\mathbb{Z}}\mathbb{Z}[\chi].$$
The following conjecture can be considered as Gross's conjecture stated in \cite{GrossConj} for the extension $F/k$:
\begin{conjecture}{\ref{conjecture0}}\label{conjecture0}
\textit{There exists a unique element $\epsilon(\chi,S)\in K_{1-2r}(O_F)_{/\mathrm{tors}}\otimes_{\mathbb{Z}}\mathbb{Z}[\chi]$  such that
$$ \rho^{r}_{F}(\epsilon(\chi,S)) =  
w_{1-r}(F^{\mathrm{ker}(\chi)})L^{'}_{S}(r,\chi^{-1}){\mid}G{\mid} e_{\chi}\in \mathbb{R}[G]\otimes_{\mathbb{Z}}\mathbb{Z}[\chi]
$$
where:
\begin{itemize}
\item[1.]$w_{1-r}(F^{\mathrm{ker}(\chi)}):={\mid}H^{0}(\mathrm{Gal}(\overline{\mathbb{Q}}/F^{\mathrm{ker}(\chi)}), \mathbb{Q}/\mathbb{Z}(1-r)){\mid}$. 
\item[2.]$e_{\chi} := \frac{1}{\mid{G}\mid}\Sigma_{\sigma\in G}\chi^{-1}(\sigma)\sigma$ is the idempotent associated with $\chi$.
\end{itemize}
}
\end{conjecture}
\begin{remark}
The unicity of the element $\epsilon(\chi,S)$ is guaranteed by the injectivity of the map $\rho^{r}_{F}$ over $K_{1-2r}(O_F)_{/\mathrm{tors}}\otimes_{\mathbb{Z}}\mathbb{Z}[\chi]$
\end{remark}
\section{Main result}
The image of $K_{1-2r}(O_F)$ by the Beilinson regulator is a complete $\mathbb{Z}$-lattice $\mathcal{T}$ in the $\mathbb{R}$-vector space $\mathbb{R}[G]$ (by Subsection \ref{BeilinsonRegSec} and Lemma \ref{lemma3.1}). We can  tensor $\mathcal{T}$ by the ring $O_p$, and define a $p$-adic regulator by extension of scalars over $O_p$
$$\rho^{r}_{F, p}: K_{1-2r}(O_F)_{/\mathrm{tors}}\otimes_{\mathbb{Z}} O_p\rightarrow \mathcal{T}\otimes_{\mathbb{Z}} O_p$$

\noindent The main result of this work is the following theorem
\begin{theorem}\label{Thm3.1}
Suppose that $p$ is a rational prime which is split in $k$ and such that $p\nmid 6$.
The following holds
\begin{itemize}
\item[(1)] The element $w_{1-r}(F^{\mathrm{ker}(\chi)})L_{S}^{'}(r,\chi^{-1}){\mid}G{\mid} e_{\chi}$ belongs to the $E$-vector space $\mathcal{T}\otimes E$.
\item[(2)] Suppose that the previous condition is fulfilled. Then, there exists a (unique) element $\epsilon(\chi, S, p)\in K_{1-2r}(O_F)\otimes O_p$ such that
$$ \rho^{r}_{F, p}(\epsilon(\chi, S, p)) =  
w_{1-r}(F^{\mathrm{ker}(\chi)})L_{S}^{'}(r,\chi^{-1}){\mid}G{\mid} e_{\chi},
$$
\end{itemize}
\end{theorem}
\begin{remark}
Theorem \ref{Thm3.1} proves that the $p$-primary part of Conjecture \ref{conjecture0} holds for all rational primes $p$ which split in $k$ and such that $p\nmid 6$.
\end{remark}
\section{The Equivariant Tamagawa Number conjecture}\label{section2}
In this section, we recall the statement of the Equivariant Tamagawa Number conjecture (the ETNc for short) for the motive $E(r)_{F}$.\\
For more details on the ETNc, the reader is invited to read the excellent survey By M. Flach in \cite{FlachSurvey}.
\subsection{Natations for the ETNc}\label{NotationsETNc}
The required tools for the statement of the ETNc are as follow
\begin{itemize}
\item For any ring $R$ we let $\mathrm{Det}_{R}$ denote the Knudsen-Mumford determinant functor \cite{KM} over $R$, which is a graded invertible (projective of constant rank 1) $R$-module.
\item We let $Y_r(F):=\prod_{\sigma\in\mathrm{Hom}(F,\mathbb{C})}(2\pi i)^{-r}\mathbb{Z}$.
\item We write $\Xi(E(r)_{F})$ for "Fontaine's fundamental line" \cite{fperrin-riou}. It is the quantity :
 $$\Xi(E(r)_{F})= \mathrm{Det}^{-1}_{E[G]}(K_{1-2r}(O_F)\otimes_{\mathbb{Z}} E)^{\#}\otimes \mathrm{Det}_{E[G]}(Y_r(F)^{+}\otimes_{\mathbb{Z}} E)^{\#},$$
 where the map $x\mapsto x^{\#}$ denotes the $\mathbb{Z}$-linear involution of the group ring $\mathbb{Z}[G]$ which satisfies $g^{\#}=g^{-1}$ for each $g\in G$.
 \item For any Galois stable lattice $T_p$ of the étale cohomology group $M_{p}:=H^{0}_{ét}(\mathrm{Spec}(F)\times_{k}\overline{k}, E_p(r))$ ($M_p$ is also the étale realization of the motive $M=E(r)_{F}$), we write $R\Gamma_{c}(O_{k,S}, T_p)$ for the complex of compact-support cohomology (e.g. \cite{nekovar}, 5.3.1). The complex $R\Gamma_{c}(O_{k,S}, T_p)$ lies in a canonical distinguished triangle
$$R\Gamma_{c}(O_{k,S}, T_p)\rightarrow R\Gamma(O_{k,S}, T_p)\rightarrow \oplus_{v\in S}R\Gamma(k_{v}, T_p).$$
where for each $v\in S$, $k_v$ denotes the completion of $k$ relatively to the place $v$.
\item We write $L_{S}(E(r)_{F},s)$ for the $S$-truncated $\mathbb{C} [G]$-valued $L$-function of the motive $E(r)_{F}$.
 It is explicitly given (with respect to the canonical identification $\mathbb{C} [G]=\prod_{\chi\in\hat{G}}\mathbb{C}$)
as
$$L_{S}(E(r)_{F},s) = (L_{S}(s+r, \chi))_{\chi\in\hat{G}}$$
We write $L_{S}^{*}(E(r)_{F},0)$ for the special value of $L_{S}(E(r)_{F},s)$ at $s=0$, which is given by
$$L_{S}^{*}(E(r)_{F},0):= (L_{S}^{*}(r, \chi))_{\chi\in\hat{G}}\in \mathbb{C}[G]^{\times}$$
where $L_{S}^{*}(r, \chi)$ is the special value of the function $L_{S}(s, \chi)$ at $s=r$ (i.e. the coefficient of the first non-zero term in its Taylor expansion). Let $S_\infty$ denote the set of archimedean places of $k$. We abbreviate $L_{S_\infty}(E(r)_{F},s)$ and $L_{S_\infty}^{*}(E(r)_{F},0)$ as $L(E(r)_{F},s)$ and $L^{*}(E(r)_{F},0)$ respectively.
\item There are two isomorphisms involved in the statement of the ETNc:
\begin{itemize}
\item The first isomorphism:
\begin{equation}\label{isoVarthetaInf}
\vartheta^{r}_{F, \infty}: \Xi(E(r)_{F})\otimes \mathbb{R}\; \xrightarrow{\;\sim\;}\; \mathbb{R}\otimes E[G]^{\#}.\;\;\;\;\;\;\;\;\;\;\;\;\;\;\;\;\;\;
\end{equation}
This is the isomorphism described in \cite{burnsGreither}, §3.2, tensored with $E$. The detailed definition of $\vartheta^{r}_{F, \infty}$ will be given in Subsection \ref{4.2.2}.
\item The second isomorphism is the $E_{p}[G]$-equivariant isomorphism
\begin{equation}\label{isoVarthetaP}
\vartheta^{r}_{F, S_p}: \Xi(E(r)_{F})\otimes E_p\; \xrightarrow{\;\sim\;} \; \mathrm{Det}_{E_p[G]}\mathrm{R\Gamma }_{c}(O_{k,S_p}, M_p)
\end{equation}
This is isomorphism (9) in \cite{burnsGreither}, tensored with $E$. The detailed definition of $\vartheta^{r}_{F, S_p}$ will be given in Subsection \ref{sub5.3}.
\end{itemize}
\end{itemize}
\subsection{Statement of the ETNc for $F/k$}
The statement of the Equivariant Tamagawa Number conjecture for the motive $E(r)_{F}$ is given as follows
\begin{conjecture}[Conjecture \ref{conjecture4.3} (the ETNc)]\label{conjecture4.3}

One has 
\begin{conjecture}[Conjecture \ref{conjecture4.3.1}.1]\label{conjecture4.3.1}{\textit{(Rationality Conjecture)}}
$$(\vartheta^{r}_{F, \infty})^{-1}(L^{*}(E(r)_{F},0)^{-1}).E[G] \supseteq \Xi(E(r)_{F})$$
\end{conjecture}

\begin{conjecture}[Conjecture \ref{conjecture4.3.4}.2]\label{conjecture4.3.4}{\textit{(the ETNc)}}\\
If the rationality conjecture (\ref{conjecture4.3.1}.1) holds then we have 
$$O_{p}[G]\vartheta^{r}_{F, S_{p}} \circ (\vartheta^{r}_{F, \infty})^{-1} (L^{*}(E(r)_{F},0)^{-1})=\mathrm{Det}_{O_{p}[G]}\mathrm{R}\Gamma_{c}(O_{k,S_p}, T_p)$$
\end{conjecture}
\end{conjecture}

\section{The ETNc and Statement (1) of Theorem \ref{Thm3.1}}\label{SectionPart1ETNc}
In this subsection, we develop first the explicit definition of the map $\vartheta^{r}_{F, \infty}$ described above in Subsection \ref{NotationsETNc}. We will need this later in our computations to prove part (1) of Theorem \ref{Thm3.1}.
\subsubsection{The isomorphism  $\vartheta^{r}_{F, \infty}$}\label{4.2.2}
\begin{definition}
Let $R$ denote a ring (with unity). For any isomorphism of finitely generated $R$-modules $\phi : V\xrightarrow{\;\sim\;} W$, we let
$$\phi_{triv}: \mathrm{Det}^{-1}_{R}(V)\otimes \mathrm{Det}_{R}(W)\xrightarrow{\;\sim\;} R,$$
obtained by the following composition of isomorphisms
$$\mathrm{Det}^{-1}_{R}(V)\otimes \mathrm{Det}_{R}(W)\stackrel{\mathrm{Det}^{-1}_{R}(\phi)\otimes \mathrm{id}}{\xrightarrow{\;\;\;\;\sim\;\;\;\;}}\mathrm{Det}^{-1}_{R}(W)\otimes \mathrm{Det}_{R}(W)\tilde{\rightarrow} R.$$
\end{definition}
The Beilinson regulator map induces an isomorphism
$$\rho^{r}_{F}: K_{1-2r}(O_F)\otimes \mathbb{R}\;\xrightarrow{\;\sim\;}\; Y_{r}(F)^{+}\otimes \mathbb{R},$$
which defines the isomorphism
$$\vartheta^{r}_{F, \infty} := (({\rho ^{r}_{F}})^{\#})_{triv} : \Xi(\mathbb{Q}(r)_{F})\otimes \mathbb{R}\; \xrightarrow{\;\sim\;}\; \mathbb{R}[G]^{\#}.$$
For the motive $E(r)_F$, extension of scalars gives
$$\vartheta^{r}_{F, \infty}: \Xi(E(r)_{F})\otimes \mathbb{R}\; \xrightarrow{\;\sim\;}\; \mathbb{R}\otimes E[G]^{\#}.$$
\begin{remark}[Remark \ref{dualisation remark}]\label{dualisation remark}
As done above, one can also define the isomorphism
$$\phi_{\widetilde{triv}}: \mathrm{Det}_{R}(V)\otimes \mathrm{Det}^{-1}_{R}(W)\xrightarrow{\;\sim\;} R,$$
and hence an isomorphism 
$$\tilde{\vartheta}^{r}_{F, \infty} := ({\rho ^{r}_{F}})_{\widetilde{triv}} : \Xi(E(r)_{F})^{\vee}\otimes \mathbb{R}\; \xrightarrow{\;\sim\;}\; \mathbb{R}\otimes E[G].$$
Since $\mathrm{Det}_{E[G]}A^{\vee}\cong \mathrm{Det}_{E[G]}^{-1}A^{\#}$ for any perfect $E[G]$-module or complex $A$ (by definition of the functor $\mathrm{Det}$, see \cite{KM})  it becomes clear that $\vartheta^{r}_{F, \infty}$ is exactly the inverse of the dual of $\tilde{\vartheta}^{r}_{F, \infty}.$\\
Hence, if $A\subseteq\Xi(E(r)_{F})\otimes \mathbb{R}$, then 
$$(\vartheta^{r}_{F, \infty}(A))^{\vee}=\tilde{\vartheta}^{r}_{F, \infty}(A^{\vee})$$
\end{remark}
\begin{proposition}\label{theorem4.4}
Suppose that the rational part of the ETNc (conjecture \ref{conjecture4.3.1}.1) holds for $E(r)_{F}$. Then
for each character $\chi\in\hat{G}$, there exists an element $\tilde{\epsilon}_{\chi}(F)\in e_{\chi}(K_{1-2r}(O_F)\otimes E)$ which verifies
$$\rho_{F}^{r}(\tilde{\epsilon}_{\chi}(F)) = L^{'}(r,\chi^{-1})$$
\end{proposition}
\begin{proof}
Conjecture \ref{conjecture4.3.1}.1  reads 
$$(\vartheta^{r}_{F, \infty})^{-1}(L^{*}(E(r)_{F},0)^{-1}).E[G] \supseteq \Xi(E(r)_{F})$$
where $$\Xi(E(r)_{F})=\mathrm{Det}^{-1}_{E[G]}(K_{1-2r}(F)\otimes E)^{\#}\otimes \mathrm{Det}_{E[G]}(Y_{r}(F)^{+}\otimes E)^{\#}.$$

\noindent By  § \ref{4.2.2}, this is equivalent to
$$ L^{*}(E(r)_{F},0)^{-1}.E[G]\supseteq ((\rho^{r}_{F})^{\#})_{triv}(\mathrm{Det}^{-1}_{E[G]}(K_{1-2r}(F)\otimes E)^{\#}\otimes\mathrm{Det}_{E[G]}(Y_{r}(F)^{+}\otimes E)^{\#})$$
Since $k$ is imaginary quadratic  we have $L^{*}(E(r)_{F},0)=L^{'}(E(r)_{F},0)$ (e.g. \cite{Den90}, $(3.4)$) and we canonically identify
$$Y_{r}(F)^{+}\otimes E\cong E[G].$$
This identification sends  for each $\sigma\in G$, seen as an embedding $\sigma:F\rightarrow \mathbb{C}$, the element $(0,..,(2\pi i)^{-r}a_{\sigma},..,(2\pi i)^{-r}a_{\tau\sigma},..,0)$ (with $a_{\sigma}=\pm a_{\tau\sigma})$ to $a_{\sigma}\sigma^{-1}$.
Hence
$$\mathrm{Det}_{E[G]}(Y_{r}(F)^{+}\otimes E)^{\#}=E[G]$$
Therefore, Conjecture \ref{conjecture4.3.1}.1 is equivalent to
$$ (L^{*}(E(r)_{F},0)^{-1})^{\#}E[G]\supseteq (\rho^{r}_{F})_{triv}(\mathrm{Det}^{-1}_{E[G]}(K_{1-2r}(F)\otimes E)\otimes E[G])$$
Let $\chi\in\hat{G}$ and $e_{\chi}$ the corresponding idempotent. Since $E$ contains all values of such characters of $G$, we get $e_{\chi}E[G] = e_{\chi}E$. In the following we use the identification $e_{\chi}E\cong E$. \\
By definition (see Subsection \ref{NotationsETNc} above)  $L^{'}(E(r)_{F},0)=\Sigma_{\chi\in\mathrm{Hom}(G,\mathbb{C}}L^{'}(r,\chi)e_{\chi}$. Conjecture \ref{conjecture4.3.1}.1 is then equivalent to the following being true for all characters of $G$
$$ L^{'}(r, \chi^{-1})^{-1}E \supseteq (\rho^{r}_{F})_{triv}(\mathrm{Det}^{-1}_{E}(e_{\chi}(K_{1-2r}(F)\otimes E))\otimes E)$$
Since $k$ is imaginary quadratic, by Beilinson regulator
$$(K_{1-2r}(O_F)\otimes E)\otimes \mathbb{R}\simeq E[G]\otimes \mathbb{R}$$
The $E$-vector space $e_{\chi}(K_{1-2r}(O_F)\otimes E)$ is then necessarily of dimension 1 and  the restriction of $(\rho^{r}_{F})_{triv}$ to the $\chi$-eigenspaces is defined as follows
\begin{align*}
(\rho^{r}_{F})_{triv}&: \mathrm{Det}^{-1}_{E}(e_{\chi}(K_{1-2r}(O_F)\otimes E))\otimes E \rightarrow E\\
&:\mathrm{Hom}(e_{\chi}(K_{1-2r}(O_F)\otimes E), E)\otimes E\stackrel{(({\rho^{r}_{F}})^{\vee})^{-1}\otimes Id_{E}}{\rightarrow} E
\end{align*}
It follows that Conjecture \ref{conjecture4.3.1}.1 is equivalent to the following being true for all characters of $G$
$$ L^{'}(r, \chi^{-1})E \subseteq \rho^{r}_{F}(e_{\chi}(K_{1-2r}(O_F)\otimes E))$$
which means there exists for each character $\chi$ of $G$ an element $\tilde{\epsilon}_{\chi}(F)\in e_{\chi}(K_{1-2r}(O_F)\otimes E)$ which verifies
$$\rho^{r}_{F}(\tilde{\epsilon}_{\chi}(F)) = L^{'}(r,\chi^{-1}).$$

\hspace*{\fill}$\Box$
\end{proof}
\begin{corollary}\label{corollary4.5}
If the rational part of the ETNc (i.e. Conjecture \ref{conjecture4.3.1}.1) holds, then part (1) of Theorem \ref{Thm3.1} is true for all rational primes $p$.
\end{corollary}
\begin{proof}
Recall that $\mathcal{T}$ denotes the complete $\mathbb{Z}$-lattice in $\mathbb{R}[G]$, which is exactly the image of $K_{1-2r}(O_F)$ by the Beilinson regulator.\\
If Conjecture \ref{conjecture4.3.1}.1  holds , then Proposition \ref{theorem4.4}, shows that the element $$w_{1-r}(F^{\mathrm{ker}(\chi)})L^{'}(r,\chi^{-1})\chi^{-1}(\sigma){\mid}G{\mid} e_{\chi},$$ belongs to $\mathcal{T}\otimes E$.\\
The proof ends once we observe that
$L^{'}_{S}(r,\chi)=L^{'}(r,\chi)\prod_{v\in S\backslash  S_\infty}(1-\chi(\mathrm{Frob}_v)\mathbf{N}v^{-r})$
and $\prod_{v\in S\backslash  S_\infty}(1-\chi(\mathrm{Frob}_v)\mathbf{N}v^{-r})\in O$.\hspace*{\fill}$\Box$
\end{proof}

\section{The ETNc and Statement (2) of Theorem \ref{Thm3.1}}\label{section5}

The main proposition of this section is the following
\begin{proposition}\label{theorem4.8}
Suppose that the ETNc (conjecture \ref{conjecture4.3.4}.2) holds for $E(r)_{F}$. Then 
\small
\begin{align*}
\rho^{r}_{F, p}(K_{1-2r}(O_{F})\otimes O_p).\mathrm{Fitt}_{O_p[G]}(&K_{-2r}(O_{F, S})\otimes O_p)=\\ 
&\varepsilon_{S_f}(r)^{\#}.L^{*}(E(r)_{F},0)^{\#}.\mathrm{Fitt}_{O_p[G]}(K_{1-2r}(O_{F})_{\mathrm{tors}}\otimes O_p)
\end{align*}
\normalsize{} 
\end{proposition}
To prove Proposition \ref{theorem4.8} we outline first some results from literature: these are preliminaries in Subsection \ref{preliminaries}. The proof of Proposition \ref{theorem4.8} is later given in Subsection \ref{proofPart2}.

\subsection{Preliminaries}\label{preliminaries}
\subsubsection{From compact-support cohomology to étale cohomology}\label{sub5.2}
In the sequel, we fix $T_p =O_{p}[G](r)$.
\begin{lemma}\label{lemma4.6}
Let $p$ be a rational prime and $S$ a finite set of places of $k$ containing the infinite places, the $p$-places and the places which ramify in $F/k$.\\
Suppose that either $p$ is odd or $k$ is totally imaginary.
There is a distinguished triangle 
\begin{equation}\label{1}
R\Gamma_{c}(O_{k,S}, T_{p})\rightarrow R\Gamma(O_{k,S}, T_{p}^{\vee}(1))^{\vee}[-3]\rightarrow (\prod_{\sigma:k\rightarrow \mathbb{C}}T_{p})^{+}[0]
\end{equation}
\end{lemma}
\begin{proof}
The distinguished triangle (\ref{1}) is the same as the one given in (\cite{burnsGreither}, §3.4). 
\hspace*{\fill}$\Box$
\end{proof}

\subsubsection{The isomorphism $\vartheta^{r}_{F, S_{p}}$}\label{sub5.3}
In this subsection, we recall the exact definition of the isomorphism $\vartheta^{r}_{F, S_p}$ (this definition is also given e.g. in \cite{burnsFlach1}, (1.17)).
\subsubsection{The $f$-cohomology}\label{subsub5.3.1}
Let $M$ be an object of the category $\mathcal{M}\mathcal{M}_F(E)$ of mixed motives over $F$ with coefficients in $E$. For such motive $M$, we let:
\begin{itemize}
\item $M^{\vee}(1)$ denote "Tate twisted dual" of $M$.
\item $H^{i}_{\mathcal{M}}(M)$ (resp. $H^{i}_{\mathcal{M}}(M^{\vee}(1))$) the motivic cohomology of $M$ (resp. of $M^{\vee}(1)$).
\item $M_B$ the Betti realization of $M$, which carries an action of complex conjugation.
\item For each rational prime number $p$, we write $M_p$ for the étale realization of $M$  $$M_{p}:=H^{0}_{ét}(\mathrm{Spec}(F)\times_{k}\overline{k}, E_p(r)).$$
\end{itemize}
Let $p$ be a rational prime number. Following Fontaine \cite{Fontaine}, for every finite place $v$ of $k$ we denote by $R\Gamma_{f}(k_{v}, M_{p})$ the complex of local unramified cohomology of $M_p$.\\
For any finite set $S$ of places of $k$ containing the set of $p$-places, and the set $S_\infty$ of infinite places, we write $R\Gamma_{f}(O_{k,S} ,M_{p})$ for the complex of global unramified cohomology of $M_p$.\\
For a detailed introduction to local and global unramified cohomology we advise the reader to consult the excellent survey by T. Nguyen Quang Do in \cite{Thong1}.\\
\begin{remark}
For later computations, we recall that, if the place $v$ is archimedean, then the complex $R\Gamma_{f}(k_{v}, M_{p})$ of local $f$-cohomology is given by
$$R\Gamma_{f}(k_{v}, M_{p}):= R\Gamma(k_{v}, M_{p})\;\;\;\;v\mid\infty$$
\end{remark}
\subsubsection{Definition of the isomorphism $\vartheta^{r}_{F, S_p}$}
Recall that $S_p$ denotes the set of places of $k$ comprised of the $p$-places and the infinite places. We let $S_{p, f}$ denote the subset of finite places of $S_p$ (i.e. of places above $p$).
\begin{proposition}\label{proposition4.2}
There is a distinguished triangle
$$R\Gamma_{c}(O_{k,S_p}, M_p)\rightarrow R\Gamma_{f}(O_{k,S_p}, M_p)\rightarrow \oplus_{v\in S_{p, f}}R\Gamma_{f}(k_{v}, M_p)\oplus \oplus_{v\in S_\infty}R\Gamma(k_{v}, M_p)$$
\end{proposition}
\begin{proof}
This is exactly the distinguished triangle (1.11) in \cite{burnsFlach1}.
\hspace*{\fill}$\Box$
\end{proof}
\begin{proposition}\label{proposition4.3}
Suppose that $p$ is odd or that $k$ is totally imaginary if $p=2$. Let $r<0$. There exists an $E_{p}[G]$-equivariant isomorphism
$$\vartheta^{r}_{F, S_p}: \Xi(E(r)_{F})\otimes E_p\; \xrightarrow{\;\sim\;} \; \mathrm{Det}_{E_p[G]}\mathrm{R\Gamma }_{c}(O_{k,S_p}, ((E(r)_{F}))_p) $$
\end{proposition}
\begin{proof}
Since we work with the motive $M:=E(r)_{F}$ and $r<0$, the fundamental line is
$$\Xi(M):=\mathrm{Det}_{E[G]}(H^{1}_{\mathcal{M}}(M^{\vee}(1))^{\vee})\otimes \mathrm{Det}^{-1}_{E[G]}(M_{B}^{+})$$
We then have the following isomorphisms
\begin{itemize}
\item $M_{B}^{+}\otimes E_{p} \stackrel{\alpha}{\xrightarrow{\;\sim\;}} (\oplus_{v\in S_{\infty}}H^{0}(k_{v},M_{p}))$
\item The Chern map isomorphism $$H^{1}_{\mathcal{M}}(M^{\vee}(1))\otimes E_{p}\stackrel{ch}{\xrightarrow{\;\sim\;}}H^{1}_{f}(O_{k,S_{p}}, M_{p}^{\vee}(1)).$$
\end{itemize}
The isomorphism $\vartheta^{r}_{F, S_p}$ is induced by these isomorphisms and the distinguished triangle of Proposition \ref{proposition4.2} on the one hand, and by the fact that,
since $r<0$, the complex $R\Gamma_{f}(k_{v}, M_p)$ is acyclic for all finite places on the other hand (e.g. \cite{burnsFlach}, §2, after Lemma 1).
\hspace*{\fill}$\Box$
\end{proof}
It is important to mention that in the definition of $\vartheta^{r}_{F, S_p}$, the module $\mathrm{Det}^{-1}_{E_{p}[G]} (\oplus_{v\in S_{p, f}}R\Gamma_{f}(k_{v}, M_p)=\otimes_{v\in S_{p, f}}\mathrm{Det}^{-1}_{E_{p}[G]} R\Gamma_{f}(k_{v}, M_p)$ maps to $E_{p}[G]$ in the following fashion (e.g. \cite{burnsFlach}, §2):\\
Write $V_{v}\rightarrow V_{v}$ for the complex $R\Gamma_{f}(k_{v}, M_p)$, then $\mathrm{Det}^{-1}_{E_{p}[G]} R\Gamma_{f}(k_{v}, M_p)$ maps to $E_{p}[G]$ via $$\mathrm{Id}_{V_{v}, triv}: \mathrm{det}^{-1}_{E_{p}[G]}V_{v}\otimes \mathrm{det}_{E_{p}[G]}V_{v} \xrightarrow{\;\sim\;} E_{p}[G].$$
and
$$\vartheta^{r}_{F, S_p}:= \mathcal{J}_{S_p}\circ(\mathrm{Det}(ch^{\vee})\otimes \mathrm{Det}^{-1}(\alpha)\otimes \big{(}\otimes_{v\in S_p}(\mathrm{Id}_{V_{v}, triv})^{-1})\big{)}.$$
\subsubsection{The isomorphism $\tilde{\vartheta}^{r}_{F, S_p}$}\label{subsub5.3.3}
There is another way to map  $\Xi(E(r)_{F})\otimes E_p$ to $\mathrm{Det}_{E_p[G]}\mathrm{R\Gamma }_{c}(O_{k,S_p}, ((E(r)_{F}))_p)$ by mapping respective cohomology groups one by one (e.g. \cite{burnsFlach1}, §1.4, the remark after (1.16). This can be achieved as follows :\\
Since $R\Gamma_{f}(k_{v}, M_p)$ is acyclic and the complex $R\Gamma_{f}(O_{k,S_p}, M_p)$ is acyclic outside degree 2 ($r<0$), the exact sequence of Proposition  \ref{proposition4.2} gives
$$\left\{
\begin{array}{lr}

 H^{1}_{c}(O_{k,S_p}, M_p)\simeq \oplus_{v\in S_{\infty}}H^{0}(k_{v}, M_p)\\
 H^{2}_{c}(O_{k,S_p}, M_p)\simeq H^{2}_{f}(O_{k,S_p}, M_p)
\end{array}
\right.$$
The isomorphism $\tilde{\vartheta}^{r}_{F, S_p}$ is induced by these isomorphisms
alongside the map
$$\mathrm{det}_{E_{p}[G]}^{-1}H^{0}_{f}(k_{v}, M_p)\otimes \mathrm{det}_{E_{p}[G]}H^{1}_{f}(k_{v}, M_p) \xrightarrow{\;\sim\;} E_{p}[G]$$
Note that the maps $\vartheta^{r}_{F, S_p}$ and $\tilde{\vartheta}^{r}_{F, S_p}$ only differ by the way $\otimes_{v\in S_{p, f}}\mathrm{Det}^{-1}_{E_{p}[G]} R\Gamma_{f}(k_{v}, M_p)$ maps to $E_{p}[G]$.

\noindent By (\cite{burnsFlach} , (11), (12)) we have
$$\tilde{\vartheta}^{r}_{F, S_p}=\varepsilon_{S_{p, f}}(r).\vartheta^{r}_{F, S_p}$$
and $\varepsilon_{S_{p, f}}(r)\in (E[G])^{\times}.$
Also, observe that $\varepsilon_{S_p}(r)$ is the inverse of the factor used in \cite{burnsFlach}, since we adopted a dual formulation of the ETNc compared to  Loc. cit.

\subsubsection{The isomorphisms $\mathcal{r}$ and $\tilde{\mathcal{r}}$}\label{sub5.4}
Let $S$ be any finite set of places of $k$, containing the set $S_p$ (of $p$-places and infinite places).
We recall the distinguished triangle (e.g. \cite{LeungKings}, Lemma 2.3)
\begin{equation*}
R\Gamma_{c}(O_{k, S}, M_{p})\rightarrow R\Gamma_{c}(O_{k, S_{p}}, M_{p}) \rightarrow
\oplus_{v\in S_{f}\backslash S_{p, f}} R\Gamma_{f}(k_{v}, M_{p})
\end{equation*}
One gets then an isomorphism
\begin{equation}
\tilde{\mathcal{r}} : \mathrm{Det}_{E_{p}[G]}R\Gamma_{c}(O_{k, S}, M_{p})\xrightarrow{\;\sim\;}\mathrm{Det}_{E_{p}[G]}R\Gamma_{c}(O_{k, S_{p}}, M_{p})
\label{delta tilde}
\end{equation}
We can also define an isomorphism
\begin{equation}
\mathcal{r} : \mathrm{Det}_{E_{p}[G]}R\Gamma_{c}(O_{k, S}, M_{p})\xrightarrow{\;\sim\;}\mathrm{Det}_{E_{p}[G]}R\Gamma_{c}(O_{k, S_{p}}, M_{p})
\label{delta}
\end{equation}
 $\mathcal{r}$ differs from $\tilde{\mathcal{r}}$ by the manner we map $\otimes_{v\in S_{f}\backslash S_{p, f}}\mathrm{Det}_{E_{p}[G]}R\Gamma_{f}(k_{v}, M_{p})$ to $E_{p}[G]$.\\
For $\mathcal{r}$ we use the map $\otimes_{v\in S_{f}\backslash S_{p, f}}\mathrm{Id}_{V_v,triv}$ introduced in the proof of Proposition \ref{proposition4.3}. For $\tilde{\mathcal{r}}$ we use the fact that the cohomology of $R\Gamma_{f}(k_{v}, M_{p})$ is acyclic for finite places. Hence:
$$\mathcal{r}=\varepsilon_{S_{f}\backslash S_{p, f}}(r).\tilde{\mathcal{r}},$$
\subsubsection{Additional considerations:}\label{Additional considerations}
\begin{itemize}
\item[1.] Let $p$ be a rational prime and $S$ a finite set of places of $k$ containing  the set $S_p$ (of infinite places and $p$-places) and the set of places which ramify in $F/k$. 
Let 
$$\tilde{\vartheta}^{r}_{F, S}:= \tilde{\mathcal{r}}^{-1}\circ \tilde{\vartheta}^{r}_{F, S_{p}}\;\;\;\;\;\;\;\;\mathrm{and}\;\;\;\;\;\;\;\;\vartheta^{r}_{F, S}:= \mathcal{r}^{-1}\circ \vartheta^{r}_{F, S_{p}}$$
where $\tilde{\mathcal{r}}$ (resp. $\mathcal{r}$) is the isomorphism (\ref{delta tilde}) (resp. (\ref{delta})) of subsection \ref{sub5.4}.

\noindent Recall that (see subsection \ref{subsub5.3.3} or e.g. \cite{burnsFlach} , (11), (12)) 
$$\tilde{\vartheta}^{r}_{F, S_{p}}=\varepsilon_{S_{p, f}}(r)\vartheta^{r}_{F, S_{p}}$$
and that (see subsection \ref{sub5.4}):
$$\mathcal{r}=\varepsilon_{S_{f}\backslash S_{p, f}}(r).\tilde{\mathcal{r}}$$
Hence
$$(\vartheta^{r}_{F, S})^{-1}=\varepsilon_{S_f}(r).(\tilde{\vartheta}^{r}_{F, S})^{-1}.$$
\item[2.] Like we did in the proof of Theorem \ref{theorem4.4}, since $k$ is imaginary quadratic, we make the following canonical identification 
$$Y_{r}(F)^{+}\cong\mathbb{Z}[G]$$
and consider the Beilinson regulator to be a map
$$\rho_{F}^{r}: K_{1-2r}(O_F)\rightarrow \mathbb{R}[G]$$
by composition with the isomorphism $\iota$ of Lemma \ref{lemma3.1}.
\item[3.] Taking in consideration the last point, we showed in the proof of Theorem \ref{theorem4.4} that 
$$e_{\chi}(\rho^{r}_{F}(K_{1-2r}(O_F))\otimes E)=L^{'}(r, \chi^{-1})e_{\chi}E$$
Thus, by the identification $L^{*}(E(r)_{F},0)=\Sigma_{\chi\in\hat{G}}L^{'}(r, \chi)e_{\chi}$ we get
$$\rho^{r}_{F}(K_{1-2r}(O_F))\otimes E=L^{*}(E(r)_{F},0)^{\#}E[G]$$
Since $L^{*}(E(r)_{F},0)\in \mathbb{R}[G]^{\times}$ we also get
$$\big((L^{*}(E(r)_{F},0)^{\#})^{-1}.\rho^{r}_{F}(K_{1-2r}(O_F))\big)\otimes E=E[G]$$
Let
$$\tilde{\mathcal{T}}:=(L^{*}(E(r)_{F},0)^{\#})^{-1}.\rho^{r}_{F}(K_{1-2r}(O_F)$$
Since $L^{*}(E(r)_{F},0)^{\#}$ is invertible, $\tilde{\mathcal{T}}$ is a $\mathbb{Z}$-order (complete lattice) of the vector space  $\mathbb{R}[G]$.\\
By extension of scalars we can form the $O_p$-order  $\tilde{\mathcal{T}}\otimes O_p$ of $E_p[G]$, which is then generated by a basis $(a_{\chi}e_\chi)_{\chi\in\hat{G}}$, where $a_\chi\in E^{\times}$ for all $\chi\in\hat{G}$. We have the following properties of $\tilde{\mathcal{T}}\otimes O_p$:
\begin{proposition}\label{prop5.5}\text{}\\
\begin{itemize}
\item[a.]$\tilde{\mathcal{T}}\otimes O_p$ is a rank-one free $O_p[G]$-module.
\item[b.] There exists a non-zero element $\delta\in O$ such that
$$\delta{\mid}G\mid.\Big(\tilde{\mathcal{T}}\otimes O_p\Big)\subseteq O_p[G]$$
 \end{itemize}
\end{proposition}
\begin{proof}
To prove $(a.)$ it suffices to observe that:
\begin{align*}
\tilde{\mathcal{T}}\otimes O_p&=\Sigma_{\chi\in\hat{G}}a_{\chi}e_\chi O_p\\
&=\frac{(\Sigma_{\chi\in\hat{G}}{\mid}G{\mid}a_{\chi}e_\chi)}{{\mid}G{\mid}}.O_p[G]
\end{align*}
and that $\Sigma_{\chi\in\hat{G}}{\mid}G{\mid}a_{\chi}e_\chi\in E[G]^{\times}$.\\
The claim $(b.)$ is straightforward.
\hspace*{\fill}$\Box$ 
\end{proof}
\begin{remark}[Remark \ref{remarkOf5.5}]\label{remarkOf5.5}
By the same argument provided in the proof of Proposition \ref{prop5.5}, one can see that $\rho^{r}_{F}(K_{1-2r}(O_F))\otimes O_p$ is also a rank-one free $O_p[G]$-module.
\end{remark}
\end{itemize}

\subsection{Proof of Proposition \ref{theorem4.8}}\label{proofPart2}

Recall that $S$ denotes a finite set of places of $k$ containing  the set $S_p$ (of infinite places and $p$-places) and the set of places which ramify in $F/k$.  The ETNc is compatible with the enlargement of the set of primes $S_p$ to $S$ (e.g. \cite{LeungKings}, Lemma 2.3).
The statement of the ETNc for the enlarged set of primes $S$ reads
$$O_{p}[G]\vartheta^{r}_{F, S} \circ (\vartheta^{r}_{F, \infty})^{-1} (L^{*}(E(r)_{F},0)^{-1})=\mathrm{Det}_{O_{p}[G]}R\Gamma_{c}(O_{k,S}, T_p),$$
and 
$$(\vartheta^{r}_{F, S})^{-1}\mathrm{Det}_{O_{p}[G]}R\Gamma_{c}(O_{k,S},T_{p})=\varepsilon_{S_f}(r).(\tilde{\vartheta}^{r}_{F, S})^{-1}R\Gamma_{c}(O_{k,S},T_{p})$$
By dualising the statement of the ETNc, and using Remark \ref{dualisation remark}, we get
\begin{align*}
O_{p}[G](\tilde{\vartheta}^{r}_{F, \infty})^{-1}(L^{*}(E(r)_{F},0)^{\#})&=\frac{1}{\varepsilon_{S_f}(r)^{\#}}(\tilde{\vartheta}^{r}_{F, S})^{\vee}\mathrm{Det}_{O_{p}[G]}R\Gamma_{c}(O_{k,S}, T_p)^{\vee}\\
&=\frac{1}{\varepsilon_{S_f}(r)^{\#}}(\tilde{\vartheta}^{r}_{F, S})^{\vee}\mathrm{Det}_{O_{p}[G]}(R\Gamma_{c}(O_{k,S}, T_p)^{\vee}[-2])
\end{align*}
where, in this case, $(\;)^\vee$ denotes $\mathrm{RHom}(\; , O_p)$.\\
By the distinguished triangle \ref{1}, one has
\begin{itemize}
\item $H^{0}(R\Gamma_{c}(O_{k,S}, T_p)^{\vee}[-2])=H^{1}(O_{F,S}, O_p(1-r))$
\item $H^{1}(R\Gamma_{c}(O_{k,S}, T_p)^{\vee}[-2])=H^{2}(O_{F,S}, O_p(1-r))\oplus (\prod_{\sigma:k\rightarrow \mathbb{C}}T_{p})^{+})^{\vee}$
\item $H^{i}(R\Gamma_{c}(O_{k,S}, T_p)^{\vee}[-2])=0$ for all $i\not = 0, 1$.
\end{itemize}
By subsection \ref{subsub5.3.3}, we know that the isomorphism $\tilde{\vartheta}^{r}_{F, S}$ is induced by the Chern isomorphism $ch$ acting on the motivic cohomology groups (i.e. duals of higher $K$-groups) and the isomorphism $\alpha$ acting on $M_{B}^{+}\otimes E\simeq (Y_{r}(F)^{+}\otimes E)^{\vee}$. The following isomorphisms (with extension of scalars over $O_p$)
\begin{itemize}
\item $H^{1}(O_{F,S}, O_p(1-r))\stackrel{ch^{\vee}}{\xrightarrow{\;\sim\;}} K_{1-2r}(O_{F})\otimes O_p$
\item $H^{2}(O_{F,S}, O_p(1-r))\oplus T_p^{\vee}\stackrel{ch^{\vee}\oplus \alpha^{\vee}}{\xrightarrow{\;\sim\;}} (K_{-2r}(O_{F, S})\otimes O_p)\oplus (Y_{r}(F)^{+}\otimes O_p)$
\end{itemize}
which hold for all $p$ since $k$ is imaginary quadratic, induce in the derived category an isomorphism (i.e. a quasi-isomorphism) between $R\Gamma_{c}(O_{k,S}, T_p)^{\vee}[-2]$ and a complex $C^{\bullet}$ which verifies
\begin{itemize}
\item $H^{0}(C^{\bullet})=K_{1-2r}(O_{F})\otimes O_p$
\item $H^{1}(C^{\bullet})=(K_{-2r}(O_{F, S})\otimes O_p)\oplus (Y_{r}(F)^{+}\otimes O_p)$
\item $H^{i}(C^{\bullet})=0$ for all $i\not = 0, 1$.
\end{itemize}
This is, for example, implicitly mentioned in (\cite{Burns}, Lemma 11.1.1). We, then, have:
$$(\tilde{\vartheta}^{r}_{F, S})^{\vee}\mathrm{Det}_{O_{p}[G]}(R\Gamma_{c}(O_{k,S}, T_p)^{\vee}[-2])=\mathrm{Det}_{O_{p}[G]}(C^{\bullet})$$
Let $\mathbb{C}_p$ be the completion of $\overline{\mathbb{Q}}_p$. The ETNc reads:
$$\varepsilon_{S_f}(r)^{\#}.L^{*}(E(r)_{F},0)^{\#}.O_p[G]=\tilde{\vartheta}^{r}_{F, \infty}\mathrm{Det}_{O_{p}[G]}(C^{\bullet}),$$
where $L^{*}(E(r)_{F},0)^{\#}.O_p[G]\subset \mathbb{C}_p[G]$. Since the Beilinson regulator induces an isomorphism
$$\rho^{r}_{F}: K_{1-2r}(O_{F})_{/\mathrm{tors}}\xrightarrow{\;\sim\;} \rho^{r}_{F}(K_{1-2r}(O_{F}))$$
which provides the isomorphism
$$K_{1-2r}(O_{F})\otimes O_p \xrightarrow{\;\sim\;} K_{1-2r}(O_{F})_{\mathrm{tors}}\otimes O_p \oplus \rho^{r}_{F}(K_{1-2r}(O_{F}))\otimes O_p$$
this induces in the derived category a quasi-isomorphism between $C^{\bullet}$ and a complex $D^{\bullet}$ that verifies 
\begin{itemize}
\item $H^{0}(D^{\bullet})=K_{1-2r}(O_{F})_{\mathrm{tors}}\otimes O_p \oplus \rho^{r}_{F}(K_{1-2r}(O_{F}))\otimes O_p$
\item $H^{1}(D^{\bullet})=(K_{-2r}(O_{F, S})\otimes O_p)\oplus (Y_{r}(F)^{+}\otimes O_p)$
\item $H^{i}(D^{\bullet})=0$ for all $i\not = 0, 1$.
\end{itemize}
and also
$$\tilde{\vartheta}^{r}_{F, \infty}\mathrm{Det}_{O_{p}[G]}(C^{\bullet}):=\mathrm{Det}_{O_{p}[G]}(D^{\bullet}).$$
We carry this process even further: Consider the maps:
\begin{itemize}
\item $\rho^{r}_{F}(K_{1-2r}(O_{F}))\otimes O_p\hookrightarrow K_{1-2r}(O_{F})_{\mathrm{tors}}\otimes O_p \oplus \rho^{r}_{F}(K_{1-2r}(O_{F}))\otimes O_p$ and
\item $\delta{\mid}G\mid.(L^{*}(E(r)_{F},0)^{\#})^{-1}\rho^{r}_{F}(K_{1-2r}(O_{F}))\otimes O_p\hookrightarrow Y_{r}(F)^{+}\otimes O_p$
\end{itemize}
where the element $\delta\in O\backslash\{0\}$ in the second map was introduced in (Proposition \ref{prop5.5}, $(b.)$). This map makes sense thanks to Proposition \ref{prop5.5}, and the fact that $Y_{r}(F)^{+}\otimes O_p\cong O_p[G]$ (see paragraph \ref{Additional considerations}, $(2.)$).

\noindent These two maps induce a morphism $f$ in the derived category, from the perfect complex 
$$Z^{\bullet}: 0\rightarrow \rho^{r}_{F}(K_{1-2r}(O_{F}))\otimes O_p\xrightarrow{0}\delta{\mid}G\mid.(L^{*}(E(r)_{F},0)^{\#})^{-1}\rho^{r}_{F}(K_{1-2r}(O_{F}))\otimes O_p\rightarrow 0$$
where the non-zero modules are placed in degrees $0$ and $1$ respectively, to the complex $D^{\bullet}$. The mapping cone of $f$ is therefore also perfect, and we have
\begin{align*}
\mathrm{Det}_{O_p[G]}D^\bullet&=\mathrm{Det}_{O_p[G]}Z^{\bullet}\otimes_{O_p[G]}\mathrm{Det}_{O_p[G]}\mathrm{Cone}(f)\\
&=\frac{L^{*}(E(r)_{F},0)^{\#}}{\delta{\mid}G\mid}.\mathrm{Det}_{O_p[G]}\mathrm{Cone}(f)
\end{align*}
We can observe now the following
\begin{itemize}
\item[1-] The ETNc holds if and only if $\mathrm{Det}_{O_p[G]}\mathrm{Cone}(f)=\delta{\mid}G\mid.\varepsilon_{S_f}(r)^{\#}.O_p[G]$.
\item[2-] The complex $\mathrm{Cone}(f)$ has finite cohomology groups as follows
\begin{itemize}
\item $H^{0}(\mathrm{Cone}(f))=K_{1-2r}(O_{F})_{\mathrm{tors}}\otimes O_p$
\item $H^{1}(\mathrm{Cone}(f))=(K_{-2r}(O_{F, S})\otimes O_p)\oplus \frac{O_p[G]}{\delta{\mid}G\mid.(L^{*}(E(r)_{F},0)^{\#})^{-1}\rho^{r}_{F}(K_{1-2r}(O_{F}))\otimes O_p}$
\item $H^{i}(\mathrm{Cone}(f))=0$ for all $i\not = 0, 1$.
\end{itemize}
\end{itemize}
Recall that (see e.g. the discussion after Lemma 3 of \cite{burnsFlach} and also paragraph $3$ of \cite{Burns}), given any perfect complex $X$ of $O_p[G]$-modules with just two nonzero cohomology groups $H^{0}(X)$ and $H^{1}(X)$, the element in the derived category which corresponds to $X$ can be represented as a Yoneda extension class
$$e(X)\in \mathrm{Ext}^{2}_{G}(H^{1}(X), H^{0}(X)),$$
a representative of which can be chosen as a $2$-extension (exact sequence)
$$0\rightarrow H^{0}(X)\rightarrow A\rightarrow B\rightarrow H^{1}(X)\rightarrow 0$$ 
with $A$ and $B$ both perfect $O_p[G]$-modules.\\
Now our proof will follow the steps of (\cite{burnsGreither2}, proof of Corollary $2$, p. 181-183). Let $\mathcal{D}^{\mathrm{p}, \mathrm{f}}(O_p[G])$ denote the category of perfect complexes of $O_p[G]$-modules with finite cohomology groups. Since 
$\mathrm{Cone}(f)\in \mathcal{D}^{\mathrm{p}, \mathrm{f}}(O_p[G])$ and the cohomology  of $\mathrm{Cone}(f)$ is nonzero only in degrees $0$ and $1$, there exists and exact sequence of $O_p[G]$-modules:
\begin{equation}\label{2-extension}
0\rightarrow H^{0}(\mathrm{Cone}(f))\rightarrow Q\xrightarrow{d} Q^{'}\rightarrow H^{1}(\mathrm{Cone}(f))\rightarrow 0
\end{equation}
which is such that both $Q$ and $Q^{'}$ are finite and of projective dimension at most $1$ over $O_p[G]$ (observe that this is exactly like the sequence $(14)$ of \cite{burnsGreither2} but with a shift), and there exists an isomorphism in $\mathcal{D}^{\mathrm{p}, \mathrm{f}}(O_p[G])$ between $\mathrm{Cone}(f)$ and the complex $Q\xrightarrow{d} Q^{'}$ (where the modules are placed in degrees $0$ and $1$, and the cohomology is identified with $H^{0}(\mathrm{Cone}(f))$ and $H^{1}(\mathrm{Cone}(f))$. Thus:
$$\mathrm{Det}_{O_p[G]}\mathrm{Cone}(f)=\mathrm{Fitt}_{O_p[G]}(Q)^{-1}\mathrm{Fitt}_{O_p[G]}(Q^{'})$$
We can apply now Lemma $5$ of \cite{burnsGreither2} to the exact sequence (\ref{2-extension}). This shows that the ETNc is equivalent to the following 
$$\delta{\mid}G\mid\varepsilon_{S_f}(r)^{\#}.\mathrm{Fitt}_{O_p[G]}(H^{0}(\mathrm{Cone}(f))^{*})=\mathrm{Fitt}_{O_p[G]}(H^{1}(\mathrm{Cone}(f))$$
where $(\;)^{*}$ denotes the Pontryagin dual, or equivalently
\small
$$\rho^{r}_{F}(K_{1-2r}(O_{F})).\mathrm{Fitt}_{O_p[G]}(K_{-2r}(O_{F, S})\otimes O_p)=\varepsilon_{S_f}(r)^{\#}.L^{*}(E(r)_{F},0)^{\#}.\mathrm{Fitt}_{O_p[G]}\Big((K_{1-2r}(O_{F})_{\mathrm{tors}}\otimes O_p)^{*}\Big)$$
\normalsize{Now}, since $(K_{1-2r}(O_{F})\otimes \mathbb{Z}_p)_{\mathrm{tors}}$ is a finite subgroup of $\mathbb{Q}_p/\mathbb{Z}_p(1-r)$, the $O_p[G]$-module $K_{1-2r}(O_{F})_{\mathrm{tors}}\otimes O_p$ is a cyclic group. We can use this fact and (Remark $9$ of \cite{burnsGreither2}, p. 183) to show that
$$\mathrm{Fitt}_{O_p[G]}\Big((K_{1-2r}(O_{F})_{\mathrm{tors}}\otimes O_p)^{*}\Big)=\mathrm{Fitt}_{O_p[G]}(K_{1-2r}(O_{F})_{\mathrm{tors}}\otimes O_p)$$.
\hspace*{\fill}$\Box$
\begin{corollary}\label{part2CorollaryETNc}
If the ETNc (i.e. Conjecture \ref{conjecture4.3}) holds, then part (2) of Theorem \ref{Thm3.1} is true for all rational primes $p$.
\end{corollary}
\begin{proof}
Let : 
\begin{itemize}
\item $A:=\rho^{r}_{F, p}(K_{1-2r}(O_{F})\otimes O_p).\mathrm{Fitt}_{O_p[G]}(K_{-2r}(O_{F, S})\otimes O_p)$
\item $B:=
\varepsilon_{S_f}(r)^{\#}.L^{*}(E(r)_{F},0)^{\#}.\mathrm{Fitt}_{O_p[G]}(K_{1-2r}(O_{F})_{\mathrm{tors}}\otimes O_p)$
\end{itemize}
Then Proposition \ref{theorem4.8} implies that $A=B$. Since ${\mid}G{\mid}e_\chi\in O[G]$ for all $\chi\in\hat{G}$, we get
$${\mid}G{\mid}e_{\chi}B\subseteq A \subseteq \rho^{r}_{F, p}(K_{1-2r}(O_{F})\otimes O_p)$$
Yet ${\mid}G{\mid}e_{\chi}B={\mid}G{\mid}.\chi(\varepsilon_{S_f}(r)^{\#}).L^{'}(r,\chi^{-1}).\big{(}e_\chi\mathrm{Fitt}_{O_p[G]}(K_{1-2r}(O_{F})_{\mathrm{tors}}\otimes O_p)\big{)}$. The proof ends by noticing that
$$w_{1-r}(F^{\mathrm{ker}(\chi)})L_{S}^{'}(r,\chi^{-1}){\mid}G{\mid} e_{\chi}\in{\mid}G{\mid}.\chi(\varepsilon_{S_f}(r)^{\#}).L^{'}(r,\chi^{-1}).\big{(}e_\chi\mathrm{Fitt}_{O_p[G]}(K_{1-2r}(O_{F})_{\mathrm{tors}}\otimes O_p)\big{)}$$
\hspace*{\fill}$\Box$
\end{proof}
\section{Conclusion of Sections \ref{SectionPart1ETNc} and \ref{section5}}
By combining Corollary \ref{corollary4.5} and Corollary \ref{part2CorollaryETNc} we get the following:
\begin{corollary}\label{ConclusionCorollary}
Suppose that the ETNc (conjecture \ref{conjecture4.3.4}) holds for $E(r)_{F}$. Then Statements (1) and (2) of Theorem \ref{Thm3.1} hold For all rational primes $p$.
\end{corollary}
\subsection{Proof of Theorem \ref{Thm3.1}}
Theorem \ref{Thm3.1} is a consequence of Corollary \ref{ConclusionCorollary} and the fact that the ETNc holds in the case of $k$ imaginary quadratic for all rational primes $p$ which split in $k$ and such that $p\nmid 6$ (by \cite{Leung}, Main Theorem, p. 4).\hspace*{\fill}$\Box$

\end{document}